\documentclass[11pt,reqno]{amsart}

\usepackage[english]{babel}

\usepackage[letterpaper,top=2cm,bottom=2cm,left=3cm,right=3cm,marginparwidth=1.75cm]{geometry}

\usepackage{amsmath, amsrefs, enumitem, mathtools, lmodern, microtype, mathrsfs, listings, tikz-cd, float}
\usepackage{amssymb}
\usepackage{graphicx}
\usepackage[colorlinks=true, allcolors=blue]{hyperref}
\usepackage{comment}
\usepackage{enumerate}

\newcommand{\N}{\mathbb{N}}

\usepackage[colorlinks=true, allcolors=blue]{hyperref}
\theoremstyle{plain}
\newtheorem{theorem}{Theorem}[section]
\newtheorem{corollary}[theorem]{Corollary}
\newtheorem{lemma}[theorem]{Lemma}

\theoremstyle{definition}
\newtheorem{defin}[theorem]{Definition}
\newtheorem{rmk}[theorem]{Remark}
\numberwithin{equation}{section}

\title{The difference of the sums of odd and even parts of restricted partitions}
\author{Kilian Rausch and Johann Stumpenhusen} 
\address{Department of Mathematics and Computer Science, Division of Mathematics, University of Cologne, Weyertal 86-90, 50931 Cologne, Germany}
\email{krausch1@uni-koeln.de}
\email{jstumpen@math.uni-koeln.de}

\begin{document}

\clubpenalty = 10000
\widowpenalty = 10000
\displaywidowpenalty = 10000

\maketitle

\begin{abstract}
For a subset $D \subset \mathbb{N},$ denote by $S\!\left(D,n\right)$ the total sum of all odd parts minus the sum of all even parts of all partitions of $n$ in which parts from $D$ do not repeat.  In this paper, we derive the generating function for $S\!\left(2\mathbb{N},n\right)$ and use it to give some congruences modulo 4. We also derive the generating function for $S\!\left(2\mathbb{N}-1,n\right)$ and give three congruences for these functions, one of which was previously proved by Garvan and Sarma. 
\end{abstract}

\section{Introduction}
A partition $\lambda=(\lambda_1,...,\lambda_s)$ of a natural number $n$ is a non-increasing sequence of natural numbers such that $\sum_{j=1}^s\lambda_j=n.$ The number of all partitions of $n$ is denoted by $p(n),$ with $p(0):=1$ by convention. The generating function for the number of partitions is given by
\begin{align*}
    \sum_{n=0}^\infty p(n)q^n = \prod_{n=1}^\infty \left(1-q^n\right)^{-1}.
\end{align*}
The expression on the right-hand side  is closely related to the $q$-Pochhammer symbol which for $a\in \mathbb{C}$ and $n \in \mathbb{N}_0 \cup \{\infty\}$ will be denoted  by $(a;q)_n:= \prod_{m=0}^{n-1} \left(1-aq^m \right)$ here and throughout.
Instead of counting the number of partitions of a certain number $n,$ we are interested in the difference between the sum of all odd parts and the sum of all even parts of a subset of all partitions of $n$. 
\begin{defin}
 Denote by $S\!\left(D,n\right)$ the sum of all odd parts minus the sum of all even parts of all partitions of $n$ in which parts from $D$ do not repeat. 
\end{defin}
In \cite{AndDast}, Andrews and Dastidar studied the function $S\!\left(\emptyset,n\right)$ (SOME$(n)$  in their notation), the difference between the sum of all odd parts minus the sum of all even parts in all  partitions of $n$, as well as  $S\!\left(\mathbb{N},n\right)$ (DSOME$(n)$ in their notation) where only partitions with distinct parts are considered. They proved the following theorem.
\begin{theorem}[\cite{AndDast}*{Theorems 3 and 9}]\label{thrmAndDast}
    The generating functions for 
    $S\!\left(\emptyset,n \right)$ and $S\!\left(\mathbb{N},n \right)$ are given by
    \begin{align*}
    \sum_{n=1}^\infty S\!\left(\emptyset,n \right)q^n&= \frac{1}{\left(q;q\right)_\infty} \sum_{n=1}^\infty \frac{q^n}{\left( 1+q^n\right)^2}, \\
    \sum_{n=1}^\infty S\!\left(\mathbb{N},n \right)q^n&= \left(-q;q\right)_\infty \sum_{n=1}^\infty \frac{(-1)^{n-1}q^n}{\left(1+q^n\right)^2}.
\end{align*}
\end{theorem}
From these, they derived some results regarding the divisibility of these partition functions evaluated at certain points.
\begin{theorem}[\cite{AndDast}*{Theorems 4, 5 and 10}]
   For $n \in \mathbb{N}_0$, it holds that
    \begin{align*}
        S\!\left(\emptyset,5n+2 \right)\equiv S\!\left(\emptyset,5n+4 \right)\equiv 0 \pmod{5} \quad and\quad S\!\left(\mathbb{N},4n \right) \equiv 0 \pmod{4}.
    \end{align*}
\end{theorem}
Recently, Bardhan and Saikia \cite{bardhan2026proofsconjecturessomen} provided proofs for a conjecture concerning $S(\emptyset, n)$ raised by Andrews and Dastidar in \cite{AndDast}. The authors of \cite{bardhan2026proofsconjecturessomen} also established some further congruences for $S(\N,n)$ modulo 2, 4 and 8.  

Instead of inspecting unrestricted and distinct partitions, one may study these differences for partitions where either all odd or all even parts are distinct and the other parts are unrestricted. 
In a recent preprint, Garvan and Sarma \cite{GarSar} considered partitions without repeated odd parts. After investigating the sum of all odd parts and the sum of all even parts in these partitions, they obtained the following result.
\begin{theorem}[\cite{GarSar}*{Theorems 1.11 and 1.12}]\label{5SDOME}
   For $n \in \mathbb{N},$ it holds that
       \begin{align*}
        S_o \!\left(5n \right)\equiv S_e\!\left(5n \right)\equiv 0 \pmod{5},
    \end{align*}
    where $S_o(n)$ (resp. $S_e(n)$) denotes the sum of all odd (resp. even) parts in all partitions of $n$ with distinct odd parts. In particular, it follows that
    \begin{align*}
         S\!\left(2\mathbb{N}-1,5n \right)\equiv 0 \pmod{5}. 
    \end{align*}
\end{theorem}
In the present paper, we display a reverse approach and offer a different proof of this result, establishing the divisibility property of $S(2\N - 1,5n)$ modulo 5 first and deriving the congruences for the individual sums from there.

Additionally, we prove a result modulo 4.
\begin{theorem}\label{4SDOME}
        For $n \in \mathbb{N},$ it holds that
    \begin{align*}
         S\!\left(2\mathbb{N}-1,4n \right)\equiv 0 \pmod{4}.
    \end{align*}
\end{theorem}

We also replicate the statement of Theorem \ref{4SDOME} for the case of partitions with distinct even parts.

\begin{theorem}\label{4SOMDE}
        For $n \in \mathbb{N} $, it holds that
    \begin{align*}
        S\!\left(2\mathbb{N},4n\right)\equiv 0 \pmod{4}.
    \end{align*}
\end{theorem}

Lastly, we derive a result regarding the parity of $S(2\mathbb{N},n)$. 
\begin{theorem}\label{SOMDEparity}
For $n \in \mathbb{N},$ it holds that
\begin{align*}
     S(2\mathbb{N},n)\equiv1 \pmod{2}
\end{align*}
    if and only if $n$ is an odd triangular number, i.e., $n=\frac{m(m+1)}{2}$ for some $m\in \mathbb{N}.$ 
\end{theorem}

This paper is structured in the following way: In Section 2 some preliminaries are given, that we will use throughout the paper. In Section 3 the generating function for $S\!\left(2\mathbb{N}-1,n\right)$ is derived and our proof for Theorem \ref{5SDOME} as well as the proof for Theorem \ref{4SDOME} are given. 
Section 4 deals with the generating function of $S\!\left(2\mathbb{N},n\right)$ and the proofs of Theorem \ref{4SOMDE} and Theorem \ref{SOMDEparity}. In the fifth and final section, we give representations for the four generating functions in terms of the Eisenstein series $E_2.$

\section*{Acknowledgements}
The authors would like to thank Kathrin Bringmann for suggesting the topic of this paper and for guidance throughout this project. 
The first author is funded by the European
Research Council (ERC) under the European Union's Horizon 2020 research and innovation programme (grant agreement No. 101001179).
\section{Preliminaries}
In this section, we give some preliminaries that we will need in the proofs in the upcoming sections. We start by giving some generating functions for expressions we will use later. Recall that by Euler's  pentagonal number theorem \cite[p. 11]{AndToP} we have that
\begin{align}\label{eulerpenta}
    \left(q;q\right)_\infty=\sum_{n=-\infty}^\infty (-1)^n q^\frac{n(3n-1)}{2},
\end{align}
and therefore, 
\begin{align*}
        \left(-q;-q\right)_\infty=\sum_{n=-\infty}^\infty (-1)^{\frac{n(3n+1)}{2}} q^\frac{n(3n-1)}{2}. 
\end{align*}
For a set $D\subset \mathbb{N},$ we denote by $p_D(n)$ the number of partitions of $n,$ where each part that is also an element of $D$ does not repeat. 
Recall, that the generating function of $p_{\mathbb{N}}(n)$, i.e., the number of partitions of $n$ into distinct parts, is given by \cite[p. 5]{AndToP}
\begin{align*}
   \sum_{n=0}^\infty p_{\mathbb{N}}(n)q^n= \prod_{n=1}^\infty \left( 1+q^n\right) = \left(-q;q\right)_\infty. 
\end{align*}
The generating function for the number of partitions of $n$ with distinct even parts, $p_{2\mathbb{N}}(n) $, is given by
\begin{align*}
    \sum_{n=0}^\infty p_{2\mathbb{N}}(n)q^n= \frac{\left(-q^2;q^2\right)_\infty}{\left(q;q^2\right)_\infty}.
\end{align*}

Next, we give an identity involving the logarithmic derivative of $\left(-q;-q\right)_\infty$. This identity will be used at a later point, when we study the generating functions of $S\!\left(2\mathbb{N}-1,n\right).$ 
\begin{lemma}\label{SDOMELemma}
It holds that
    \begin{align*}
    \sum_{n=1}^\infty \frac{(-q)^n}{\left(1-\left(-q\right)^n \right)^2}= -q \frac{\partial}{\partial q} \textup{Log}\!\left(\left(-q;-q\right)_\infty \right). 
\end{align*}
\end{lemma}
\begin{proof}
    A direct calculation gives
\begin{align*}
    -q \frac{\partial}{\partial q} \textup{Log}\!\left(\left(-q;-q\right)_\infty \right) &=  -q \frac{\partial}{\partial q} \textup{Log}\!\left( \prod_{n=1}^\infty \left(1-(-q)^n\right)\right)
 = -q \frac{\partial}{\partial q} \sum_{n=1}^\infty \textup{Log}\left(1-(-q)^n \right)
 \\&=  \sum_{n=1}^\infty \frac{n (-q)^{n}}{1-(-q)^n}
= \sum_{n=1}^\infty n \sum_{m=1}^\infty (-q)^{nm}
 \\&= \sum_{m=1}^\infty\sum_{n=1}^\infty n((-q)^m)^n
 = \sum_{m=1}^\infty \frac{(-q)^m}{(1-(-q)^m)^2}. \qedhere
\end{align*}
\end{proof}
Lastly, we state and prove a combinatorial result which we will need in the second to last section.  For $n \in \mathbb{N}_0$ we denote by $r(n)$ the number of integer tuples $(k,j)$ such that $n= \frac{k(3k-1)}{2} + j(3j-1),$ i.e.,
    \begin{align*}
    r(n):= \#\! \left\{(k,j)\in \mathbb{Z}^2 \,\bigg|\, \frac{k(3k-1)}{2}+j(3j-1)=n\right\}.
    \end{align*} 
\begin{lemma}\label{R(n)}
    Then the parity of $r(n)$ is given by
    \begin{align*}
        r(n) \equiv \begin{cases}
            1 \pmod{2}& n=\frac{m(m+1)}{2} \text{ for some }  m\in \mathbb{N}_0,
            \\ 0\pmod{2} & \text{otherwise}.
        \end{cases} 
    \end{align*}
\end{lemma}
\begin{proof}
    The argument presented in this proof was generated by OpenAI's artificial intelligence model GPT-6 Astra. The authors assume full responsibility for the correctness of the presented proof.

    We denote by 
    \begin{align*}
        F(q):= \sum_{n=-\infty}^\infty q^{\frac{n(3n-1)}{2} }
    \end{align*}
    and with this we immediately get that
    \begin{align*}
        \sum_{n=0}^\infty r(n)q^{n}= F(q)F\!\left(q^2\right).
    \end{align*}
    By Euler's pentagonal number theorem (\ref{eulerpenta}), we have
    \begin{align*}
        \prod_{m=1}^\infty \left(1-q^m\right) = \sum_{n=-\infty}^\infty(-1)^n  q^{\frac{n(3n-1)}{2} }\equiv \sum_{n=-\infty}^\infty q^{\frac{n(3n-1)}{2} } = F(q) \pmod{2}.
    \end{align*}
   Combining this with the previous equations, we get 
    \begin{align*}
        \sum_{n=0}^\infty r(n)q^n \equiv \prod_{m=1}^\infty \left(1-q^{m}\right)\!\left(1-q^{2m}\right) \pmod{2}.
    \end{align*}
    
    Recall the Jacobi triple product identity (see for example \cite[6.21]{KoecherKrieg2025}) for $|x|<1$ and $y\neq 0$, 
    \begin{align*}
        \sum_{n=-\infty}^\infty x^{n^2}y^{n}=\prod_{m=1}^\infty \left(1-x^{2m}\right)\!\left(1+x^{2m-1}y\right)\!\left(1+x^{2m-1}y^{-1}\right).
    \end{align*}
   Substituting $x=y=q^\frac{1}{2}$, we obtain
    \begin{align*}
        \sum_{n =-\infty}^\infty q^{\frac{n(n+1)}{2}} &= \prod_{m=1}^\infty \left(1-q^{m}\right)\!\left(1+q^{m}\right)\!\left(1+q^{m-1}\right) 
        \\&= 2 \prod_{m=1}^\infty \left(1-q^{m}\right)\!\left(1+q^{m}\right)^2 .
    \end{align*}
   As
    \begin{align*}
        \frac{n(n+1)}{2}= \frac{n^2+n}{2} =\frac{-(n+1)(-(n+1)+1)}{2},
    \end{align*}
    combining the terms for $n$ and $-(n+1)$ in the sum, we get that
    \begin{align*}
        \sum_{n =0}^\infty q^{\frac{n(n+1)}{2}}&=  \prod_{m=1}^\infty \left(1-q^{m}\right)\!\left(1+q^{m}\right)^2.
    \end{align*}
    With this, we get 
    \begin{align*}
        \sum_{n =0}^\infty q^{\frac{n(n+1)}{2}}&=  \prod_{m=1}^\infty \left(1-q^{m}\right)\!\left(1+q^{m}\right)^2 
        \\&\equiv \prod_{m=1}^\infty \left(1-q^{m}\right)\!\left(1-q^{2m}\right) \pmod{2}
        \\&= \sum_{n=0}^\infty r(n) q^n. 
    \end{align*}
    From this we conclude that $r(n)$ is odd if and only if $n$ is a triangular number, i.e., $n= \frac{m(m+1)}{2}$.
\end{proof}
With these preliminaries, we are now ready to prove our main results. 
\section{Proofs for the claims regarding $S\!\left(2\mathbb{N}-1,n\right)$}
We start this section by deriving the generating function for $S\!\left(2\mathbb{N}-1,n\right)$.
\begin{theorem}\label{geneSDOME}
    For $n \in \mathbb{N}$, the generating function for $S\!\left(2\mathbb{N}-1,n\right)$ is given by
    \begin{align*}
        \sum_{n=1}^\infty S\!\left(2\mathbb{N}-1,n\right)q^n= -\frac{\left(-q;q^2\right)_\infty}{\left(q^2;q^2\right)_\infty} \sum_{n=1}^\infty \frac{\left(-q\right)^n}{\left(1-(-q)^n\right)^2} .
    \end{align*}
\end{theorem}
\begin{proof} A direct calculation gives
       \begin{align*}
        \sum_{n=1}^\infty S\!\left(2\mathbb{N}-1,n\right)q^n &=  \frac{\partial }{\partial z} \bigg|_{z=1}\prod_{n=1}^\infty  \frac{1+(zq)^{2n-1}}{1-\left(\frac{q}{z}\right)^{2n}}
     \\&= \sum_{n=1}^\infty \left( \frac{(2n-1)q^{2n-1}}{1-q^{2n}} - \frac{2nq^{2n}\!\left(1+q^{2n-1} \right)}{\left(1-q^{2n}\right)^2} \right) \prod_{\substack{m=1\\m\neq n}}^\infty \frac{1+q^{2m-1}}{1-q^{2m}}
     \\&= \frac{\left(-q;q^2\right)_\infty}{\left(q^2;q^2\right)_\infty}  \sum_{n=1}^\infty \frac{(2n-1)q^{2n-1}}{1+q^{2n-1}} - \frac{2nq^{2n}}{1-q^{2n}}
     \\&= -\frac{\left(-q;q^2\right)_\infty}{\left(q^2;q^2\right)_\infty} \sum_{n=1}^\infty \frac{n\left(-q\right)^n}{1-\left(-q\right)^n }
    = -\frac{\left(-q;q^2\right)_\infty}{\left(q^2;q^2\right)_\infty} \sum_{n=1}^\infty n \sum_{m=1}^\infty \left( \left(-q\right)^n\right)^m
    \\ &= -\frac{\left(-q;q^2\right)_\infty}{\left(q^2;q^2\right)_\infty} \sum_{m=1}^\infty \frac{\left(-q\right)^m}{\left(1-\left(-q\right)^m\right)^2} \qedhere.
    \end{align*}
\end{proof}

\begin{rmk}
    Garvan and Sarma \cite{GarSar} computed the generating functions for $S_o(n)$ and $S_e(n)$ separately, from which one may then also obtain the generating function in Theorem \ref{geneSDOME}.
\end{rmk}

The generating function of $S\!\left(2\mathbb{N}-1,n\right)$ is an important ingredient in the next proof.  
\begin{proof}[Proof of Theorem \textup{\ref{5SDOME}}]
 This proof follows the structure of the proof of Theorem 4 presented by Andrews and Dastidar in \cite{AndDast}. In particular, it is reversed compared to the strategy employed by Garvan and Sarma in \cite{GarSar}.
 Using  Theorem \ref{geneSDOME} first, then Lemma \ref{SDOMELemma} and then (2.1) we get that
 \begin{align*}
        \sum_{n=1}^\infty S\!\left(2\mathbb{N}-1,n\right)q^n &= -\frac{\left(-q;q^2\right)_\infty}{\left(q^2;q^2\right)_\infty} \sum_{n=1}^\infty \frac{\left(-q\right)^n}{\left(1-(-q)^n\right)^2} 
        \\ &=\frac{\left(-q;q^2\right)_\infty}{\left(q^2;q^2\right)_\infty}  q \frac{\partial}{\partial q} \textup{Log}\!\left(\left(-q;-q\right)_\infty \right)
    \\&= \frac{\left(-q;q^2\right)_\infty}{\left(q^2;q^2\right)_\infty} q \frac{\partial}{\partial q} \textup{Log}\! \left(\sum_{k=-\infty}^\infty (-1)^{\frac{k(3k+1)}{2}} q^\frac{k(3k-1)}{2}\right)
    \\&=  \frac{\left(-q;q^2\right)_\infty} {\left(q^2;q^2\right)_\infty \!\left(-q;-q\right)_\infty} \sum_{k=-\infty}^\infty (-1)^{\frac{k(3k+1)}{2}} \frac{k(3k-1)}{2} q^{\frac{k(3k-1)}{2}} 
    \\&= \frac{1}{\left(q^2;q^2\right)_\infty^2} \sum_{k=-\infty}^\infty (-1)^{\frac{k(3k+1)}{2}} \frac{k(3k-1)}{2} q^{\frac{k(3k-1)}{2}}
    \\ &=   \frac{\left(q^2;q^2\right)_\infty^3}{\left(q^2;q^2\right)_\infty^5} \sum_{k=-\infty}^\infty (-1)^{\frac{k(3k+1)}{2}} \frac{k(3k-1)}{2} q^{\frac{k(3k-1)}{2}}.
 \end{align*}
Modulo 5, the last expression is congruent to 
\begin{align*}
 \frac{1}{\left(q^{10};q^{10}\right)_\infty} \sum_{m=0}^\infty (-1)^m (2m+1)q^{m^2+m}  \sum_{k=-\infty}^\infty (-1)^{\frac{k(3k+1)}{2}} \frac{k(3k-1)}{2} q^{\frac{k(3k-1)}{2}}.
\end{align*}
Reducing the exponents modulo  $5$, one checks when $m^2+m + \frac{k(3k-1)}{2}$ is divisible by 5. This is only the case if both $m^2+m$ as well as $ \frac{k(3k-1)}{2}$ are divisible by 5. In these cases, the coefficient $ \frac{k(3k-1)}{2}$ of the right sum is divisible by 5 and thus, the coefficients of $q^{5n}$ for any $n \in \mathbb{N}$ vanish modulo 5. For the other claim, we have 
    \begin{align*}
        {S}_o(5n)+{S}_e(5n)=5n p_{2\mathbb{N}-1}(5n)\equiv 0 \pmod{5}.
    \end{align*}
    By Theorem \ref{5SDOME} we also have
    \begin{align*}
        {S}_o(5n)-{S}_e(5n)=S\!\left(2\mathbb{N}-1,5n\right)\equiv 0 \pmod{5}.
    \end{align*}
    Adding and subtracting these two congruences  yields
    \begin{align*}
         2{S}_o(5n) \equiv 0 \pmod{5} \text{ and   }  2 {S}_e(5n)\equiv 0 \pmod{5}.
    \end{align*} This yields the result. 
\end{proof}
\begin{proof}[Proof of Theorem \textup{\ref{4SDOME}}]
      Our argument follows the structure of the proof of Theorem 8 in \cite{AndDast}. Recall that from the proof of Theorem \ref{5SDOME}  we know that
      \begin{align*}
           \sum_{n=1}^\infty S\!\left(2\mathbb{N}-1,n\right)q^n =  \frac{1}{\left(q^2;q^2\right)_\infty^2} \sum_{k=-\infty}^\infty (-1)^{\frac{k(3k+1)}{2}} \frac{k(3k-1)}{2} q^{\frac{k(3k-1)}{2}}.
      \end{align*}
We start by isolating the even powers of $q,$ writing
\begin{align*}
    \frac{1}{\left(q^2;q^2\right)_\infty^2}&\sum_{k=-\infty}^\infty (-1)^{\frac{k(3k+1)}{2}} \frac{k(3k-1)}{2} q^{\frac{k(3k-1)}{2}} \frac{1+(-1)^{\frac{k(3k-1)}{2}}}{2}
    \\&= \frac{1}{\left(q^2;q^2\right)_\infty^2}\sum_{k=-\infty}^\infty \left((-1)^{\frac{k(3k+1)}{2}}+(-1)^k\right) \frac{k(3k-1)}{4} q^{\frac{k(3k-1)}{2}}
    \\&= \frac{1}{\left(q^2;q^2\right)_\infty^2}\sum_{n=-\infty}^\infty  2n(12n-1) q^{\frac{4n(12n-1)}{2}}- (4n+3)(6n+4) q^{\frac{(4n+3)(12n+8)}{2}},
\end{align*}
as 
\begin{align*}
  (-1)^\frac{k(3k+1)}{2} +(-1)^k = \begin{cases}
        2 & k\equiv 0 \pmod{4},\\
        -2 &  k\equiv 3 \pmod{4},\\
        0 & \text{otherwise.}
    \end{cases}
\end{align*}
Replacing $q^2 $ by $q,$ we get
\begin{align*}
    &\phantom{\equiv} \frac{1}{\left(q;q\right)_\infty^2}\sum_{n=-\infty}^\infty  2n(12n-1)q^{n(12n-1)}- (4n+3)(6n+4) q^{(4n+3)(3n+2)}
    \\ &\equiv \frac{2}{\left(q;q\right)_\infty^2} \sum_{n=-\infty}^\infty -nq^{n(12n-1)} - (9n+6)q^{(3n+2)(4n+3)} \pmod{4}.
\end{align*}
Lastly, we have to show that the coefficients of the even powers of $q$ are divisible by 4. This is equivalent to showing that the coefficients of the even powers of $q$ in 
\begin{align*}
    \frac{1}{\left(q;q\right)_\infty^2} \sum_{n=-\infty}^\infty -nq^{n(12n-1)} - (9n+6)q^{(3n+2)(4n+3)}
\end{align*}
are divisible by 2. Evaluating modulo 2, we get
\begin{align*}
    \frac{1}{\left(q;q\right)_\infty^2} &\sum_{n=-\infty}^\infty -nq^{n(12n-1)} - (9n+6)q^{(3n+2)(4n+3)} 
  \\ &\equiv \frac{1}{\left(q^2;q^2 \right)_\infty} \sum_{n=-\infty}^\infty nq^{n(12n-1)}+nq^{(3n+2)(4n+3)}  \pmod{2}.
\end{align*}
Here the exponents of $q$ are even if and only if $n$ is even.  
\end{proof}

\section{Proofs of the claims regarding $S\!\left(2\mathbb{N},n\right)$}
As we did in the last section, we start with deriving the generating function for $S\!\left(2\mathbb{N},n\right)$.
\begin{theorem}\label{geneSOMDE}
    For $n \in \mathbb{N}_0,$ the generating function for $S\!\left(2\mathbb{N},n\right)$ is given by
    \begin{align*}
        \sum_{n=1}^\infty S\!\left(2\mathbb{N},n\right)q^n=  \frac{\left(-q^2;q^2\right)_\infty}{\left(q;q^2\right)_\infty} \sum_{n=1}^\infty \frac{q^n}{\left(1-(-q)^n\right)^2}. 
    \end{align*}
\end{theorem}
\begin{proof}
    A direct calculation gives
       \begin{align*}
        \sum_{n=1}^\infty S\!\left(2\mathbb{N},n\right)q^n &=  \frac{\partial }{\partial z} \bigg|_{z=1}\prod_{n=1}^\infty  \frac{1+\left(\frac{q}{z} \right)^{2n}}{1-\left(zq\right)^{2n-1}}
        \\&= \left( \sum_{n=1}^\infty  \frac{-2nq^{2n}}{1-q^{2n-1}} + \frac{\left(1+q^{2n}\right)(2n-1)q^{2n-1}}{\left(1-q^{2n-1} \right)^2} \right)  \prod_{\substack{m=1\\m\neq n}}^\infty \frac{1+q^{2m}}{1-q^{2m-1}}
        \\&= \frac{\left(-q^2;q^2\right)_\infty}{\left(q;q^2\right)_\infty} \sum_{n=1}^\infty \frac{-2nq^{2n}}{1+q^{2n}}+ \frac{(2n-1)q^{2n-1}}{1-q^{2n-1}}
        \\&= -\frac{\left(-q^2;q^2\right)_\infty}{\left(q;q^2\right)_\infty} \sum_{n=1}^\infty \frac{2n\left(-q\right)^{2n}}{1+ \left(-q\right)^{2n}}+ \frac{(2n-1)\left(-q\right)^{2n-1}}{1+ \left(-q\right)^{2n-1}}
        \\&= -\frac{\left(-q^2;q^2\right)_\infty}{\left(q;q^2\right)_\infty} \sum_{n=1}^\infty  \frac{n\left(-q\right)^n}{1+\left(-q\right)^n} = \frac{\left(-q^2;q^2\right)_\infty}{\left(q;q^2\right)_\infty} \sum_{n=1}^\infty  n \sum_{m=1}^\infty \left(- \left(-q\right)^n \right)^m
        \\&= \frac{\left(-q^2;q^2\right)_\infty}{\left(q;q^2\right)_\infty} \sum_{m=1}^\infty (-1)^m \sum_{n=1}^\infty n\left(\left(-q\right)^m \right)^n
        =\frac{\left(-q^2;q^2\right)_\infty}{\left(q;q^2\right)_\infty} \sum_{m=1}^\infty \frac{q^m}{\left(1-\left(-q\right)^m\right)^2}. \qedhere     
        \end{align*}
\end{proof}
\begin{proof}[Proof of Theorem \textup{\ref{4SOMDE}}]
    The proof of Theorem \ref{4SOMDE} follows the structure of the proof of Theorem 9 in \cite{AndDast}.
    We have by (2.3) that 
    \begin{align*}
        \sum_{n=1}^\infty np_{2\mathbb{N}}(n)q^n &= q \frac{\partial }{\partial q}\frac{\left(-q^2;q^2\right)_\infty}{\left(q;q^2\right)_\infty}
          \\ &=q \sum_{m=1}^\infty \left(\frac{2mq^{2m-1}}{1-q^{2m-1}} + \frac{(2m-1)\left(1+q^{2m}\right) q^{2m-2}}{\left(1-q^{2m-1}\right)^2} \right)\prod_{\substack{n=1\\n\neq m}}^\infty \frac{1+q^{2n}}{1-q^{2n-1}}
    \\&=\frac{\left(-q^2;q^2\right)_\infty}{\left(q;q^2\right)_\infty} \sum_{m=1}^\infty \frac{mq^m}{1+(-q)^m}. 
    \end{align*}
Rewriting the inner sum, we get
\begin{align*}
    \sum_{m=1}^\infty \frac{mq^m}{1+(-q)^m} &= -\sum_{m=1}^\infty \frac{(-1)^mm\left(-\left(-q\right)^m\right)}{1-\left(-(-q)^m\right)} 
= -\sum_{m=1}^\infty (-1)^mm \sum_{n=1}^\infty \left(-(-q)^m\right)^n
\\&=-\sum_{n=1}^\infty (-1)^n \sum_{m=1}^\infty m \left(-(-q)^n\right)^m
    = -\sum_{n=1}^\infty (-1)^n\frac{-(-q)^n}{\left(1+(-q)^n\right)^2}
\\&=\sum_{n=1}^\infty \frac{q^n}{\left(1+(-q)^n\right)^2}.
\end{align*}
Combining this gives     
\begin{align*}
    \sum_{n=1}^\infty S\!\left(2\mathbb{N},n\right)q^n&= \frac{\left(-q^2;q^2\right)_\infty}{\left(q;q^2\right)_\infty} \sum_{n=1}^\infty \frac{q^n}{\left(1-(-q)^n\right)^2}
    \\&\equiv \frac{\left(-q^2;q^2\right)_\infty}{\left(q;q^2\right)_\infty} \sum_{n=1}^\infty \frac{q^n}{\left(1+(-q)^n\right)^2} \pmod{4}
    \\&= \sum_{n=1}^\infty n p_{2\mathbb{N}}(n)q^n .
\end{align*}
Since $4n\cdot p_{2\mathbb{N}}(4n)$ is divisible by 4, so the claim holds. 
\end{proof}
In order to prove Theorem \ref{SOMDEparity}, we need an auxiliary lemma.
\begin{lemma}\label{lemmatriangle}
   For $n \in \N,$ it holds that
    \begin{align*}
        S(2\mathbb{N},n)\equiv \! \left| \Bigl\{ (k,j)\in \mathbb{Z}: n= \frac{k(3k-1)}{2}+j(3j-1) \wedge k \equiv 1,2 \pmod{4}\Bigr\} \right|\pmod{2}.
    \end{align*}
\end{lemma}
\begin{proof}
  Our arguments follow the idea of the proof of Theorem 4.1  of \cite{bardhan2026proofsconjecturessomen}. Using the congruence from the  previous proof, we get
\begin{align*}
    \sum_{n=1}^\infty S(2\mathbb{N},n)q^n&=- \frac{\left(-q^2;q^2\right)_\infty}{\left(q;q^2\right)_\infty} \sum_{n=1}^\infty \frac{q^n}{\left(1-(-q)^n\right)^2}
    \\&\equiv -\frac{\left(-q^2;q^2\right)_\infty}{\left(q;q^2\right)_\infty} \sum_{n=1}^\infty \frac{q^n}{\left(1+(-q)^n\right)^2} \pmod{2}
    \\&= \sum_{n=1}^\infty n p_{ed}(n)q^n .
\end{align*}
From the previous proof, we also know that 
\begin{align*}
    \sum_{n=1}^\infty np_{ed}(n)q^n &= q \frac{\partial }{\partial q}\frac{\left(-q^2;q^2\right)_\infty}{\left(q;q^2\right)_\infty}.
\end{align*}
We can rearrange this to
\begin{align*}
    \frac{\left(-q^2;q^2\right)_\infty}{\left(q;q^2\right)_\infty}&= \prod_{n=0}^\infty\frac{\left(1+q^{2n+2}\right)}{\left(1-q^{2n+1}\right)} \overset{\text{(mod 2)}}{\equiv} \prod_{n=0}^\infty\frac{\left(1-q^{2n+2}\right)}{\left(1-q^{2n+1}\right)} = \prod_{n=0}^\infty\frac{\left(1-q^{n+1}\right)\!\left(1+q^{n+1}\right)}{\left(1-q^{2n+1}\right)}
    \\&= \prod_{n=1}^\infty \left(1+q^{n}\right) \prod_{\substack{n=1\\n \text{ even}}}^\infty\left(1-q^{n}\right) \overset{\text{(mod 2)}}{\equiv} \left(q;q\right)_\infty \!\left(q^2;q^2\right)_\infty.
\end{align*}

 Using the pentagonal number theorem (\ref{eulerpenta}),
we obtain that
\begin{align*}
    \left(q;q\right)_\infty \!\left(q^2;q^2\right)_\infty= \sum_{k=-\infty}^\infty (-1)^k q^\frac{k(3k-1)}{2}\sum_{k=-\infty}^\infty (-1)^j q^{j(3j-1)}= \sum_{k,j=-\infty}^\infty (-1)^{j+k} q^{\frac{k(3k-1)}{2}+ j(3j-1)}
\end{align*}
from which we conclude that 
\begin{align*}
    q\frac{d}{dq} \! \left(\left(q;q\right)_\infty \!\left(q^2;q^2\right)_\infty \right) &=  \sum_{k,j=-\infty}^\infty (-1)^{j+k} \!\left(\frac{k(3k-1)}{2}+j(3j-1) \right) q^{\frac{k(3k-1)}{2}+j(3j-1)}
    \\&\equiv \sum_{\substack{j,k=-\infty\\k\equiv 1,2 \pmod{4}}}^\infty  q^{\frac{k(3k-1)}{2}+j(3j-1)} \pmod{2},  
\end{align*}
as $j(3j-1) $ is always even and the term $ \frac{k(3k-1)}{2}$ is even if and only if $k\equiv 0,3 \pmod{4}$. Combining all this, we are left with
\begin{align*}
    \sum_{n=1}^\infty S(2\mathbb{N},n)q^n&\equiv \sum_{\substack{k,j=-\infty\\k\equiv 1,2 \pmod{4}}}^\infty  q^{\frac{k(3k-1)}{2}+j(3j-1)} \pmod{2}.
\end{align*}
Comparing coefficients yields the desired result. 
\end{proof}
\begin{rmk}
    From the definition of $S(2\N,n)$ it follows directly, that $S(2\N,2n)\equiv 0 \pmod{2}.$ In that case, $n=\frac{k(3k-1)}{2}+ j(3j-1)$ has no solutions with $k \equiv 1,2 \pmod{4}.$
\end{rmk}
With this lemma, we now have all ingredients needed to provide a proof for Theorem \ref{SOMDEparity}:
\begin{proof}[Proof of Theorem \textup{\ref{SOMDEparity}}]
    By definition of $S(2\N, n)$ we have for even $n$ that $S(2\N,n)\equiv0 \pmod{2}$.   Next consider the case where $n$ is odd. By the previous lemma we have
\begin{align*}
        S(2\N,n)&\equiv \left| \Bigl\{ (k,j)\in \mathbb{Z}: n= \frac{k(3k-1)}{2}+j(3j-1) \wedge k \equiv 1,2 \pmod{4}\Bigr\} \right|\pmod{2} 
        \\ &\equiv \left| \Bigl\{ (k,j)\in \mathbb{Z}: n= \frac{k(3k-1)}{2}+j(3j-1) \Bigr\} \right| \pmod{2},  
    \end{align*}
    as $n= \frac{k(3k-1)}{2}+j(3j-1)$ is odd if and only if $k \equiv 1,2 \pmod{4}.$
Now using Lemma \ref{R(n)} we get
    \begin{align*}
S(2\N,n) \equiv r(n) \equiv \begin{cases} 
1 & n=\frac{m(m+1)}{2}\\
0 & \text{otherwise}
\end{cases} \pmod{2}.
    \end{align*}
So $S(2\N,n)$ is odd if and only if $n$ is an odd triangular number.
\end{proof}
We can also give an alternative proof for Theorem \ref{SOMDEparity} that does not rely on characterizing the parity of 
\begin{align*}
   \left| \left\{ (k,j)\in \mathbb{Z}: n= \frac{k(3k-1)}{2}+j(3j-1) \right\} \right|  .
\end{align*}
\begin{proof}[A second proof of Theorem \textup{\ref{SOMDEparity}}]
  From the proofs of Lemma \ref{lemmatriangle} and Lemma \ref{R(n)} we know that
\begin{align*}
    \frac{\left(-q^2;q^2\right)_\infty}{\left(q;q^2\right)_\infty}\equiv\left(q;q\right)_\infty \!\left(q^2;q^2\right)_\infty \equiv \left(q;q\right)_\infty^3
    \equiv \sum_{n=0}^\infty q^{\frac{n(n+1)}{2}} \pmod{2}.
\end{align*}
Using this and the result
\begin{align*}
     \sum_{n=1}^\infty S(2\N,n)q^n\equiv q\frac{d}{dq} \!\left(  \left(q;q\right)_\infty\!\left(q^2;q^2\right)_\infty\right)  \pmod{2}
\end{align*}
from the proof of Theorem \ref{4SOMDE}, we conclude that
\begin{align*}
    \sum_{n=1}^\infty S(2\N,n)q^n\equiv q\frac{d}{dq}  \!\left(\left(q;q\right)_\infty\left(q^2;q^2\right)_\infty\right) &\equiv \sum_{n=0}^\infty \frac{n(n+1)}{2}q^n  \pmod{2}.
\end{align*}
Now comparing the coefficients gives the desired result. 
\end{proof}
A direct consequence of Theorem \ref{SOMDEparity} is the following corollary.
\begin{corollary}
    Let $n \in \N$ be odd. The number of partitions of $n$ into distinct even parts and unrestricted odd parts is odd if and only if $n$ is a triangular number. 
\end{corollary}
\begin{rmk}
This result is in no way a new result. In one of his papers from 1972, Andrews \cite{AndrewGaus} states an identity, which he attributes to Gauss, from which the more general result, that the number of partitions of a number $n$ into distinct even parts and unrestricted odd parts is odd if and only if $n$ is a triangular number, is a direct consequence. One therefore confidently assumes that Gauss already knew about the more general result. 
\end{rmk}
\section{Some identities for the generating functions} 
In this section, we give some identities for the generating functions considered in this paper, in terms of the Eisenstein series $E_2$. As usual, we define the divisor sum $\sigma$ for $n \in \N$,
\begin{align*}
    \sigma(n):=\sum_{d|n}  d
\end{align*}
and with this
\begin{align*}
    E_2(\tau):= 1- 24\sum_{k=1}^\infty \sigma(k)q^k,
\end{align*}
where $q:= e^{2\pi i \tau}$ and $\tau \in \mathbb{H}:=\{z \in \mathbb{C}| \text{Im}(z)>0 \}$.

We start with the generating function of $S(\emptyset,n)$. From \cite{AndDast} we know that
\begin{align*}
    \sum_{n=1}^\infty \frac{q^n}{\left(1+q^n\right)^2}= \sum_{m=1}^\infty \sum_{n=1}^\infty (-1)^{n-1}nq^{mn}= \sum_{k=1}^\infty q^k \sum_{d|k}(-1)^{d-1}d.
\end{align*}
We closely examine the inner sum and obtain
\begin{align*}
    \sum_{d|k}(-1)^{d-1} d =  \sum_{2\nmid d|k} d - \sum_{2|d|k} d=  \sigma(k)- 2\sum_{2|d|k} d .
\end{align*}
The rightmost sum vanishes, if $k$ is odd. If $k$ is even, we obtain
\begin{align*}
    2\sum_{2|d|k} d =4\sum_{\frac{d}{2}|\frac{k}{2}} \frac{d}{2}= 4\sigma\!\left(\frac{k}{2}\right). 
\end{align*}
Employing Theorem \ref{thrmAndDast}, this gives
\begin{align*}
    \sum_{n=1}^\infty S(\emptyset , n) q^n &= \frac{1}{\left(q;q\right)_\infty} \! \left(\sum_{k=1}^\infty \sigma(k) q^k  - 4\sum_{\substack{k=1\\2|k}}^\infty \sigma\!\left(\frac{k}{2} \right)q^k\right)
\\ &= \frac{1}{\left(q;q\right)_\infty} \!\left( \sum_{k=1}^\infty \sigma(k)q^k- 4\sum_{k=1}^\infty  \sigma(k) q^{2k}\right)
\\&= \frac{1}{24\left(q;q\right)_\infty} (-E_2(\tau)+4E_2(2\tau)-3).
\end{align*}

Next, we consider the generating function of $S(\mathbb{N},n).$  By \cite{AndDast} we have 
\begin{align*}
    \sum_{n=1}^\infty \frac{(-1)^{n-1}q^n}{\left(1+q^n\right)^2} = \sum_{m=1}^\infty\sum_{n=1}^\infty (-1)^{n+m}nq^{mn}= \sum_{k=1}^\infty q^k \sum_{d|k} (-1)^{d+\frac{k}{d}}d.
\end{align*}
We again examine the inner sum. If $k$ is odd, we find
\begin{align*}
    \sum_{d|k} (-1)^{d+\frac{k}{d}}d= \sum_{d|k}d=\sigma(k).
\end{align*}
For the case that $k$ is even, rewrite the sum as a convolution 
\begin{align*}
    \sum_{d|k} (-1)^{d+\frac{k}{d}}d&= \left(2u-I_0 \right) I_1 * \left(2u-I_0\right)(k)
    \\&= (2uI_1*2u)(k)-(2uI_1*I_0)(k)-(I_1*2u)(k)+(I_1*I_0)(k),
\end{align*}
where $(f*g)(n):= \sum_{d|n} f(d) g\left(\frac{n}{d}\right),$  $I_\ell(n):=n^\ell$ and 
\begin{align*}
    u(n):=\mathbf{1}_{\{n \in \N: 2\mid n\}}:=\begin{cases}
        1 &  \text{if } 2|n,\\
        0 & \text{else.}
    \end{cases} 
\end{align*} 
Evaluating these convolutions gives:
\begin{align*}
    (2uI_1*2u)(k)&= 4\sum_{\substack{d|k\\2|d,2|\frac{k}{d}}}d = \begin{cases}
        8\sum_{d|\frac{k}{4}}d=8\sigma\!\left(\frac{k}{4}\right) & \text{if } 4\mid k,\\
        0 & \text{if } 4\nmid k,
    \end{cases}\\
    -(2uI_1*I_0)(k)&=-2\sum_{\substack{2|d|k}}d= -4 \sum_{\frac{d}{2}|\frac{k}{2}} \frac{d}{2} = -4\sigma\!\left(\frac{k}{2}\right), \\
    -(I_1*2u)(k)&= -2\sum_{\substack{d\mid k\\2|\frac{k}{d}}} d=-2\sum_{d\mid  \frac{k}{2}}d=-2\sigma\!\left( \frac{k}{2}\right),
\intertext{and lastly,}
    (I_1*I_0)(k)&= \sum_{d|k}d =\sigma(k).
\end{align*}
Again, by Theorem \ref{thrmAndDast} we obtain
\begin{align*}
    \sum_{n=1}^\infty S\!\left(\mathbb{N},n\right)q^n &= \left(-q;q\right)_\infty \!\left( \sum_{k=1}^\infty \sigma(k) q^k -6 \sum_{k=1}^\infty \sigma(k) q^{2k} +8 \sum_{k=1}^\infty \sigma(k) q^{4k}\right) 
    \\&= \frac{\left(-q;q\right)_\infty}{24} \left(-E_2(\tau)+6E_2(2\tau)-8E_2(4\tau)+3\right).
\end{align*}
Recall that both $-E_2(\tau)+2E_2(2\tau)$ and $E_2(\tau)-4E_2(4\tau)$ are modular forms of weight 2 and level 4, so
\begin{align*}
       \sum_{n=1}^\infty S\!\left(\mathbb{N},n\right)q^n &= \left(-q;q\right)_\infty \!\left(F(\tau)+\frac{1}{8}\right),
\end{align*}
where $F(\tau):= E_2(\tau)+6E_2(2\tau)-8E_2(4\tau),$ is a modular form on $\Gamma_0(4).$

Lastly, we consider the generating functions of $S(2\mathbb{N}-1,n)$ and $S(2\mathbb{N},n).$ From our previous calculations in the proof of Lemma \ref{SDOMELemma} we know that 
\begin{align*}
    \sum_{n=1}^\infty \frac{(-q)^n}{\left(1-(-q)^n\right)^2} &= \sum_{m=1}^\infty \sum_{n=1}^\infty m(-q)^{mn},
\end{align*}
substituting $k=mn,$ we obtain
\begin{align*}
     \sum_{m=1}^\infty \sum_{n=1}^\infty m(-q)^{mn}= \sum_{k=1}^\infty (-q)^k \sum_{d|k}d=\sum_{k=1}^\infty \sigma(k) \left(-q \right)^k = -\frac{E_2\!\left(\tau+ \frac{1}{2}\right)-1}{24},
\end{align*}
which, in combination with Theorem \ref{geneSDOME}, gives
\begin{align*}
    \sum_{n=1}^\infty S(2\mathbb{N}-1,n)q^n = \frac{ \left(-q;q^2\right)_\infty \!\left( E_2\! \left( \tau + \frac{1}{2} \right)-1 \right)}{24 \left(q^2;q^2\right)_\infty}. 
\end{align*}
Arguing as before, we also obtain
\begin{align*}
     \sum_{n=1}^\infty \frac{q^n}{\left(1-(-q)^n\right)^2}& = \sum_{m=1}^\infty \sum_{n=1}^\infty (-1)^nm(-q)^{mn} = \sum_{k=1}(-q)^k \sum_{d|k} (-1)^\frac{k}{d}d\\
     &=-\sum_{k=1}^\infty \sigma(k) \left(-q \right)^k + 2\sum_{k=1}^\infty \sigma(k)q^{2k},
\end{align*}
which, together with Theorem \ref{geneSOMDE}, yields
\begin{align*}
    \sum_{n=1}^\infty S(2\mathbb{N},n)q^n= \frac{\left(-q^2;q^2\right)_\infty \!\left(E_2 \!\left(\tau + \frac{1}{2} \right)-2E_2(2\tau) +1\right)}{24\left(q;q^2\right)_\infty} .
\end{align*}

\end{document}